\documentclass[12pt]{article}
\usepackage{amssymb,amsmath,amsthm}

\newtheorem{thm}{Theorem}
\newtheorem{prop}[thm]{Proposition}
\newtheorem{cor}[thm]{Corollary}

\newtheorem{claim}[thm]{Claim}

\newtheorem{conj}[thm]{Conjecture}

\usepackage{indentfirst}
\title{On the tree packing conjecture}

\author{J\'ozsef Balogh\thanks{Department of Mathematical Sciences,
University of Illinois at Urbana-Champaign, Urbana, Illinois 61801, USA {\tt jobal@illinois.edu}. Research supported by NSF CAREER Grant DMS-0745185, UIUC Campus Research Board Grant 13039, and OTKA Grant K76099.} \and Cory Palmer\thanks{Department of Mathematical Sciences,
University of Illinois at Urbana-Champaign, Urbana, Illinois 61801, USA {\tt ctpalmer@illinois.edu}. Research supported by OTKA Grant NK78439.}}

\begin{document}

\maketitle

\begin{abstract}
The  Gy\'arf\'as tree packing conjecture states that any set of $n-1$ trees $T_{1},T_{2},\dots, T_{n-1}$ such that $T_i$ has $n-i+1$ vertices pack into $K_n$.
We show that $t=\frac{1}{10}n^{1/4}$ trees $T_1,T_2,\dots, T_t$ such that $T_i$ has $n-i+1$ vertices pack into $K_{n+1}$ (for $n$ large enough). 
We also prove that any set of $t=\frac{1}{10}n^{1/4}$ trees $T_1,T_2,\dots, T_t$ such that no tree is a star and $T_i$ has $n-i+1$ vertices pack into $K_{n}$ (for $n$ large enough).
Finally, we prove that $t=\frac{1}{4}n^{1/3}$ trees $T_1,T_2,\dots, T_t$ such that $T_i$ has $n-i+1$ vertices pack into $K_n$ as long as each tree has maximum degree at least $2n^{2/3}$ (for $n$ large enough).
One of the main tools used in the paper is the famous spanning tree embedding theorem of Koml\'os, S\'ark\"ozy and Szemer\'edi \cite{kssz-paper}.
\end{abstract}

\section{Introduction}
A set of (simple) graphs $G_1,G_2, \dots, G_n$ are said to \emph{pack} into a graph $H$ if $G_1,G_2, \dots, G_n$ can be found as pairwise edge-disjoint subgraphs in $H$. 
In this paper we are concerned with
the case when each $G_i$ is a tree and $H$ is a complete graph on $n$ vertices, denoted by $K_n$.
The famous tree packing conjecture (TPC) posed by Gy\'arf\'as (see \cite{GyLe}) states:

\begin{conj}\label{theconjecture}
Any set of $n-1$ trees $T_n,T_{n-1},\dots, T_{2}$ such that $T_i$ has $i$ vertices pack into $K_n$.
\end{conj}

Bollob\'as suggested a weakening of TPC in the Handbook of Combinatorics \cite{handbook}:

\begin{conj}
For every $k \geq 1$ there is an $n(k)$ such that if $n \geq n(k)$, then
any set of $k$ trees $T_1,T_2,\dots, T_k$ such that $T_i$ has $n-i+1$ vertices pack into $K_n$.
\end{conj}

A number of partial results concerning the TPC are known. The first results are by Gy\'arf\'as and Lehel \cite{GyLe} who proved that
the TPC holds with the additional assumption that all but two of the trees are stars. Roditty \cite{Ro} confirmed TPC in the case when all but three trees are stars (see also \cite{GeKePa}).
Gy\'arf\'as and Lehel also showed that the TPC is true if each tree is either a path or a star. A second proof of this statement is by Zaks and Liu \cite{ZaLi}\footnote{An incorrect version of this 
proof appears in \cite{Gems}.}.
Bollob\'as \cite{Bo} showed that
any set of $k-1$ trees $T_k,T_{k-1},\dots, T_2$ such that $T_i$ has $i$ vertices pack into $K_n$ if $k \leq \frac{\sqrt{2}}{2}n$. 
Bollob\'as also noted that the bound on $k$ can be increased to $\frac{\sqrt{3}}{2}n$
if we assume the Erd\H os-S\'os conjecture is true (see \cite{Er46}).
Packing many large trees seems to be difficult. Hobbs, Bourgeois and Kasiraj \cite{HoBoKa} showed that any three trees $T_n,T_{n-1},T_{n-2}$ such that $T_i$ has $i$ vertices pack into $K_n$. 
A series of papers by Dobson \cite{Do1,Do2,Do3} concerns packing trees into $K_n$ with restrictions on the structure of each tree.
 % Hobbs showed diameter 3. Fishbourne extended this. Also showed TPC for n<10. Straight extended the two results of Gy-L. 

Instead of packing trees into the complete graph, a number of papers have examined packing trees into complete bipartite graphs. 
Hobbs, Bourgeois and Kasiraj \cite{HoBoKa} conjectured that $n-1$ trees $T_{n},T_{n-1},\dots, T_{2}$ such that $T_i$ has $i$ vertices pack into the complete bipartite graph
$K_{n-1, \lfloor n/2 \rfloor}$. The conjecture is true if each of the trees is a star or path (see Zaks and Liu \cite{ZaLi} and Hobbs \cite{Ho}). 
Yuster \cite{Yu} showed that $k-1$ trees $T_{k},T_{k-1},\dots, T_2$ such that $T_i$ has $i$ vertices pack into $K_{n-1,\lfloor n/2 \rfloor}$ if $k \leq \lfloor \sqrt{5/8}n \rfloor$ (improving the previously 
best-known bound on $k$ by Caro and Roditty \cite{CaRo}). 
Various generalizations of the tree packing conjecture were investigated by Gerbner, Keszegh and Palmer \cite{GeKePa}. Recently, B\"ottcher, Hladk\'y, Piguet and Taraz \cite{Hl} proved
an asymptotic version of the tree packing conjecture for trees with bounded maximum degree.

Notation will be standard (following e.g.\ \cite{Bo-book}). A vertex of a graph of degree $1$ is a
\emph{leaf}. A set of leaves in a graph are \emph{independent} if the
neighbors of the leaves are pairwise disjoint (i.e.\ the edges incident
to a set of independent leaves form a matching).
A \emph{leaf edge} is an edge incident to a vertex of degree $1$.
We denote by $G[H]$ the induced graph of $G$ on the vertex set $H$ and
by $G[H_1,H_2]$ the induced bipartite graph of $G$ with classes
$H_1,H_2$. The \emph{neighborhood} of a set of vertices $X$ is the set of vertices not in $X$ with a neighbor in $X$ (i.e.\ neighborhoods are not considered closed).
The maximum degree of a graph $G$ is denoted by $\Delta(G)$; the minimum degree by $\delta(G)$. For the sake of brevity, the set of vertices
of a graph $G$ will also be denoted by $G$. 

The set of first $k$ integers is denoted by $[k]$. For $a \in A$ we will write $A-a$ for $A-\{a\}$. 

Clearly a set of graphs $G_1,G_2,\dots, G_k$ pack into $H$ if there is a $k$-edge-coloring of $H$ where the graph induced by the edges of color $i$ contains a $G_i$. Generally we will pack a set of trees by starting with an uncolored complete graph and $k$-coloring the edges in a series of steps. Thus we call an edge \emph{uncolored} if it has not yet received a color.

We will suppress all integer part notation. We note that in an effort to make the proofs easier many of the multiplicative constants are allowed to be larger than is necessary.
Our main results are Theorems \ref{main1} and \ref{main2}:

\begin{thm}\label{main1}
Let $n$ be sufficiently large and let $t=\frac{1}{4}n^{1/3}$.
Then the trees $T_1,T_2,\dots, T_t$ pack into $K_n$ if for each $i$ we have $|T_i| = n-i+1$ and
at least one of the following holds:
\begin{enumerate}
\itemsep0em
\item[(1)] $T_i$ has a set of at most $n^{1/3}$ vertices such that the union of their neighborhoods contains at least $n^{2/3}$ leaves.
\item[(2)] $T_i$ has at least $n^{2/3}$ independent leaves.
\end{enumerate}
\end{thm}

Any tree with maximum degree at least $2n^{2/3}$ must satisfy (1) or (2) from Theorem~\ref{main1} thus we have the following corollary.

\begin{cor}\label{main1-cor}
Let $n$ be sufficiently large and let $t=\frac{1}{4}n^{1/3}$.
If $T_1,T_2,\dots, T_t$ are trees such that $|T_i| = n-i+1$ and
$\Delta(T_i) \geq 2n^{2/3}$ for every $i$, then $T_1,T_2,\dots,T_t$ pack into
$K_n$.
\end{cor}

If we let the complete graph have one more vertex than allowed by
Conjecture \ref{theconjecture}, then we can pack many trees without conditions on their structure.

\begin{thm}\label{main2}
Let $n$ be sufficiently large and $t=\frac{1}{10}n^{1/4}$.
If $T_1,T_2,\dots, T_t$ are trees such that $|T_i| = n-i+1$ for every $i$,
then $T_1,T_2,\dots,T_t$ pack into $K_{n+1}$.
\end{thm}

Eliminating a single case from the proof of Theorem~\ref{main2} gives the following proposition.

\begin{prop}\label{main2-cor}
Let $n$ be sufficiently large and $t=\frac{1}{10}n^{1/4}$.
If $T_1,T_2,\dots, T_t$ are trees such that $|T_i| = n-i+1$ and $T_i$ is not a star for each $i$,
then $T_1,T_2,\dots,T_t$ pack into $K_{n}$.
\end{prop}

The remainder of the paper is is organized as follows. In Section~\ref{prelim} we will prove some preliminary claims that will help with the proofs of both theorems. Sections~\ref{proof1} and \ref{proof2} concern the proofs of Theorems~\ref{main1} and \ref{main2}.

%paper outline and proof outline

\section{Preliminaries}\label{prelim}

Before proving Theorems \ref{main1} and \ref{main2}, we will
need some preparation. Koml\'os, S\'ark\"ozy and Szemer\'edi \cite{kssz-paper}
proved the following.

\begin{thm}\label{kssz}
Let $ \delta > 0$ be given. Then there exist constants $c$ and $n_0$
with the following properties. If $ n \geq n_0 $, $T$ is a tree on $n$ vertices with
$ \Delta (T) \leq { cn /\log n}$, and $G$ is a graph on $n$ vertices with
$ { \delta (G)} \geq (\frac{1}{2}+ { \delta })n $, then $T$ is a
subgraph of $G$.
\end{thm}

The following corollary is an immediate consequence of Theorem~\ref{kssz} (we fix $\delta = 1/6$ here).
%For a function $g(n)$ recall that $\omega(g(n))$ refers to the class of functions $f(n)$ such that  $f(n)/g(n) \rightarrow \infty$ as $n \rightarrow \infty$.

\begin{cor}\label{kssz-cor}
Let $n$ be sufficiently large and let $t = t(n)$ be such that $t(n)/\log n \rightarrow \infty$ as $n \rightarrow \infty$.
If $T_1,T_2,\dots, T_t$ are forests of order at most $n$ and $\Delta(T_i) < \frac{1}{3} n/t$, then
$T_1,T_2,\dots, T_t$ pack into $K_n$.
\end{cor}

Corollary~\ref{kssz-cor} will allow us to pack into $K_n$ the forests
that remain after removing vertices of ``high'' degree from each tree
$T_i$. We will also need the following easy claims.
The first follows from an application of the greedy algorithm.

\begin{claim}\label{stars-claim}
Fix $k,a,b$ such that $k < a < b$. Let $G$ be a bipartite graph with classes $A = \{v_1,v_2,\dots,v_a\}$ and $B$ such that $|B|=b$ and the degree of each vertex in $A$ is at least $b-k$.
Then for any non-negative integers $c_1,c_2,\dots,c_a$ such that
$\sum_{i=1}^a c_i\leq b-k$ we can pack a star forest into $G$ such that for all $i$ each vertex $a_i$ is the center of a star with exactly $c_i$ leaves.
\end{claim}

%\begin{proof}
%We proceed by induction on $k$. For $k=1$ the claim is trivial. So let $k>1$ and assume the statement of the claim for $k-1$.
%Clearly $c_k \leq |B|-1$, so we can color $c_k$ edges incident to $a_k$ with color $k$. Let $A' = A - \{a_k\}$ and $B'$ be $B$ minus the vertices
%incident to edges of color $k$ and consider $G' \subset G$ as the bipartite graph with classes $A'$ and $B'$. It is easy to see the conditions of the claim
%hold for $G'$ and thus we can complete the coloring of $G$ by induction.
%\end{proof}

\begin{claim}\label{matching-claim}
Fix $a,b,k$ such that $4k^2 < a < b-k$. Let $G$ be a graph resulting from the removal of $k$ forests from a complete bipartite graph $K_{a,b}$. Then $G$ contains a matching of size $a-k$.
\end{claim}

\begin{proof}
Let $A,B$ be the two vertex classes of $K_{a,b}$ such that $|A|=a$ and $|B|=b$.
By the defective version of Hall's theorem, if there is no matching in $G$ of size $a-k$, then there exists a nonempty set $S \subset A$ with neighborhood $N(S) \subset B$ such that $|S| -k > |N(S)|$.  
Clearly there is no edge between $S$ and $B- N(S)$. These non-edges form a subgraph of the union of the forest(s) removed from $K_{a,b}$. Such a subgraph has average degree less than $2k$. 
Thus either $|S| < 2k$ or $|B|-|N(S)|<2k$.

First assume $|S| < 2k$.
Immediately we have $|N(S)| < |S|-k < k$.
Observe that there are $|S|(|B|-|N(S)|)$ non-edges between $S$ and $B- N(S)$. Furthermore, these non-edges
form a subgraph of the forest(s) removed from $K_{a,b}$, so there are at most $k(|S|+|B|-|N(S)|-1)$ such non-edges. 
Therefore $|S|(|B|-|N(S)|) \leq k(|S|+|B|-|N(S)|-1)$. 
Solving for $|B|$ gives $|B| \leq \frac{k(|S|-|N(S)|-1) + |S||N(S)|}{|S|-k} \leq 4k^2$, a contradiction.

Now assume $|B|-|N(S)|<2k$. This gives $|B|-2k < |N(S)| < |S|-k$ and thus $|S| > |B|-k > a$, a contradiction.
\end{proof}

%The following claim is a consequence of the well-know fact that any regular bipartite graph has a perfect matching.

\begin{claim}\label{matching2}
Fix $a$ and $k$ such that $a \geq 2k$ and let $G$ be a graph resulting from the removal of $k$ matchings from a $K_{a,a}$. Then $G$ has a perfect matching. 
\end{claim}

\begin{proof}
If $G$ does not have a perfect matching, then by Hall's theorem there exists a nonempty set $S$ in one of the partite classes with neighborhood $N(S)$ such that
$|S| > |N(S)|$. First observe that a vertex is not in $N(S)$ if every edge between it and $S$ has been removed. Thus $|S| \leq k$. Furthermore,
we have that the number of non-edges with an endpoint in $S$ is at least $|S|(a-|N(S)|)$ and at most $|S|k$. 
Simplifying this inequality gives $k \leq a-k \leq |N(S)| < |S|$, a contradiction.
\end{proof}

As stated in the introduction, Conjecture \ref{theconjecture} is true when each tree is either a path or a star.
By examining the packing of any set of paths and stars $T_n,T_{n-1},\dots, T_2$ where $T_i$ has $i$ vertices into $K_n$ as given by Zaks and Liu \cite{ZaLi} it is
easy to see that the set of endpoints of the paths with at least $\frac{2}{3}n$ vertices are pairwise vertex-disjoint in $K_n$. This gives the following helpful claim.

\begin{claim}\label{stars-paths}
Let $T_1,T_2,\dots, T_k$ be trees such that each $T_i$ is either a path or a star and $|T_i| = 3k-i+1$. Then
there is a $k$-edge-coloring of $K_{3k}$ such that for each $i$, the edges of color $i$ span $T_i$ and 
the set of endpoints of the paths are pairwise vertex-disjoint in $K_{3k}$.
\end{claim}

\section{Proof of Theorem~\ref{main1}}\label{proof1}

\begin{proof}%[Proof of Theorem~\ref{main1}]
Throughout the proof we will assume that $n$ is sufficiently large for
the appropriate inequalities to hold. For ease of notation put $h =  \frac{3}{4}n^{2/3}$ and note that $8t=2n^{1/3}$, thus $8t^2 = \frac{1}{2}n^{2/3} = h - \frac{1}{4}n^{2/3}$.
We partition the vertex set of $K_n$ into three parts of order $n-h-8t$, $h$, and $8t$.
We will refer to the three parts as $K_{n-h-8t}$, $K_h$, and $K_{8t}$.

If $T_i$ satisfies condition (1) from the statement of Theorem~\ref{main1} then we call it \emph{type~I}, otherwise it satisfies condition (2) and is called \emph{type~II}. We will
partition the trees into parts corresponding to the partition of $K_n$.

\medskip

{\bf Partition of type~I trees:} 
For each tree $T_i$ of type~I, define $H_i$ to be the union of the set of vertices of degree greater than $n^{2/3}$ in $T_i$ 
and a set of at most $n^{1/3}$ vertices in $T_i$ such that the union of their neighborhoods contains at least $n^{2/3}$ leaves
and an arbitrary set of vertices such that $|H_i|= 8t = 2n^{1/3}$.

Partition $T_i$ into three parts: $H_i$, a set of $h - (i-1)$ leaves in the neighborhood of $H_i$ and 
the remaining $n-(i-1)-|H_i|-(h - (i-1)) = n-h-8t$ vertices, denoted by $F_i$.

\medskip

{\bf Partition of type~II trees:}
 For each tree $T_i$ of type~II, let $Y_i$ be a set of $t-1$ independent leaves and denote the neighborhood of $Y_i$ by $X_i$ (thus $|X_i| = t-1$).
Define $H_i$ to be the union of $X_i$ and the set of vertices of degree greater than $n^{2/3}$ in $T_i$
and an arbitrary set of vertices in $T_i- Y_i$ such that $|H_i| = 8t = 2n^{1/3}$.

Partition $T_i$ into four parts: $H_i$, $Y_i$, a set of $h - |Y_i| - (i-1)$ independent leaves that are not adjacent to $H_i$ (there are at most $|H_i|= 2n^{1/3} = 8t$ independent leaves adjacent to $H_i$), denoted by $L_i$, and the remaining $n-(i-1)-|H_i|-|Y_i|-(h - |Y_i| - (i-1)) = n - h -8t$ vertices, denoted by $F_i$.

\medskip

{\bf Packing into $K_{n}$:} 
To pack the trees into $K_n$ we first pack each $F_i$ into $K_{n-h-8t}$ and each $H_i$ into $K_h$ at the same time. Then we will embed the remaining parts of the trees one-by-one starting with type~I trees and finishing with type~II trees.
Recall that each
$F_i$ is a forest on $n-h-8t$ vertices and $\Delta(F_i) \leq n^{2/3} < \frac{4}{3}(n^{2/3}-\frac{3}{4}n^{1/3} - 2) = \frac{1}{3}(n-\frac{3}{4}n^{2/3} - 2n^{1/3})/t  = \frac{1}{3}(n-h-8t)/t$.
Thus by Corollary~\ref{kssz-cor} the forests $F_1,F_2,\dots, F_t$ pack into $K_{n-h-8t}$. In
other words there is an edge-coloring of $K_{n-h-8t}$ such that for each $i$ the edges
of color $i$ contain $F_i$.

We can pack each $H_i$ vertex-disjointly into $K_h$ as $\sum_{i=1}^t |H_i| = 8t^2 = \frac{1}{2}n^{2/3} < h = \frac{3}{4}n^{2/3}$. Because the sets $H_i$ are disjoint in $K_h$,
for each $i$ the edges $T_i[H_i,F_i]$ in $K_n$ can be colored with $i$.

We will now complete the packing of the trees starting with type~II followed by type~I.
We pack the trees of a given type from largest to smallest i.e.\ when completing the packing $T_i$ we may assume that all trees of the same type among $T_1,T_2,\dots, T_{i-1}$ have already been packed into $K_{n}$. 
A set of vertices in $K_{n}$ are called \emph{finished (in color $i$)} if each vertex has as many incident edges of color $i$ as its degree in $T_i$. Otherwise it is \emph{unfinished (in color $i$)}.

\medskip

{\bf Packing of type~II trees:}
For each type~II tree $T_i$ let $N_i$ be the set of neighbors of the independent leaves $L_i$, so $|N_i| = h - |Y_i| - (i-1)$. 
Observe that $N_i \subset F_i$, therefore the vertices in $N_i$ are already packed in $K_{n-h-8t}$.
To complete the packing of $T_i$ we need to find a matching of uncolored edges between $N_i \cup X_i$
 and a set of vertices with no incident edge of color $i$ such that $N_i \cup X_i$ is covered. These edges will represent the leaf edges between $N_i \cup X_i$ (which are already packed into $K_n$) and
$L_i \cup Y_i$ (which have not yet been packed into $K_n$).
Thus coloring the edges of this matching with $i$ will complete the packing of $T_i$.

Consider the bipartite graph between $N_i$ and $\cup_{j=1}^{i-1} H_j$ in $K_n$. Observe that at this point any color from $[i-1]$ may have been used on the edges of this bipartite graph and that the edges in a single color class form a forest. Thus we have removed at most $i-1$ forests from a complete bipartite graph with class sizes:
\begin{align*}
|\cup_{j=1}^{i-1} H_j|  = (i-1)8t \geq 4(i-1)^2  
\end{align*}
and
\begin{align*}
|N_i| &= h- |Y_i| - (i-1) > h - 2t  = \frac{3}{4}n^{2/3} - \frac{1}{2}n^{1/3}\\
      & > \frac{1}{2}n^{2/3} + (i-1) = 8t^2 + (i-1)\\
      & > (i-1)8t + (i-1) = |\cup_{j=1}^{i-1} H_j| + (i-1).
\end{align*}
So we may apply Claim~\ref{matching-claim} to the bipartite graph between $N_i$ and $\cup_{j=1}^{i-1} H_j$ with $i-1$ forests removed. 
Therefore there is a matching of uncolored edges between $N_i$ and $\cup_{j=1}^{i-1} H_j$ that misses only $i-1$ vertices of $\cup_{j=1}^{i-1} H_j$. 
Color this matching with $i$. The $i-1$ vertices missed by the matching will never have an incident edge of color $i$.

Now we consider the bipartite graph between the unfinished vertices of $N_i$ and $K_h-\cup_{j=1}^{i} H_j$ in $K_n$. Observe that at this point any color except $i$ may
have been used on the edges between these two classes and that the edges in a single color class form a forest.
Thus we have removed at most $t-1$ forests from a complete bipartite graph with class sizes:
$$
|K_h-\cup_{j=1}^{i} H_j|  = h - i8t > h - 8t^2 \\
			          = \frac{3}{4}n^{2/3} - \frac{1}{2}n^{2/3}\\
				  = \frac{1}{4}n^{2/3} = 4t^2 > 4(t-1)^2
$$
and the number of unfinished vertices in $N_i$ i.e.\ 
\begin{align*}
|N_i| - |\cup_{j=1}^{i-1} H_j| &=  h- |Y_i| - (i-1) - (i-1)8t \\
			       &= h - (t-1) - (i-1) - i8t + 8t\\
			       &> h - i8t + (t-1) = |K_h-\cup_{j=1}^{i} H_j| + (t-1). 
\end{align*}
So we may apply Claim~\ref{matching-claim} to the bipartite graph between the unfinished vertices of $N_i$ and $K_h-\cup_{j=1}^{i} H_j$ with $t-1$ forests removed.
Therefore there is a matching of uncolored edges between the unfinished vertices of $N_i$ and $K_h-\cup_{j=1}^{i} H_j$ that misses only $t-1$
vertices of $K_h-\cup_{j=1}^{i} H_j$. Color this matching with $i$. 

Now we consider the set of $t-1$ vertices in $K_h-\cup_{j=1}^{i} H_j$ not incident to an edge of color $i$. 
We can embed $Y_i$ into these vertices and color the corresponding edges between $X_i$ and $Y_i$ (these are leaf edges of $T_i$) 
with $i$ as no edges between $H_i$ and $K_h-\cup_{j=1}^{i} H_j$ have been colored.

Finally, there remains $2n^{1/3}=8t$ unfinished vertices in $N_i$. For each color $j \in [i-1]$ there is at most one edge of color $j$ incident to each vertex in $K_{8t}$. 
Thus by Claim~\ref{matching2} there is a matching of uncolored edges between the unfinished vertices of $N_i$ and $K_{8t}$. Coloring this matching with $i$ completes the packing of $T_i$.

{\bf Packing of type~I trees:} 
To complete the packing of $T_i$ we need to color edges incident to $H_i$ that correspond to the $h-(i-1)$ leaf edges removed from $T_i$. Recall that
these leaf edges form a star forest with each center vertex in $H_i$.
Observe that each vertex in $H_i$ is incident to at most one edge of color $j \in [i-1]$ with an endpoint in $\cup_{j=1}^{i-1} H_j$. 
Furthermore, all edges between $H_i$ and $K_h - \cup_{j=1}^{i} H_j$ are uncolored as are all edges between $H_i$ and $K_{8t}$. 
Thus the bipartite graph with classes $H_i$ and $(K_h - H_i) \cup K_{8t}$ is such that each vertex in $H_i$ is incident to at most $i-1$ colored edges.
Therefore each vertex in $H_i$ is incident to at least $|(K_h - H_i) \cup K_{8t}| - (i-1) = h-8t+8t-(i-1) = h-(i-1)$ uncolored edges between $H_i$ and $(K_h - H_i) \cup K_{8t}$.
So we may apply Claim~\ref{stars-claim} with $k=i-1$ to the bipartite graph of uncolored edges between $H_i$ and $(K_h - H_i) \cup K_{8t}$
to find the appropriate star forest removed from $T_i$. Coloring this star forest with $i$ completes the packing of $T_i$.
\end{proof}

\section{Proof of Theorem~\ref{main2}}\label{proof2}
The proof of Theorem~\ref{main2} follows the general structure of the proof of Theorem~\ref{main1}. 
However, we must introduce a new type of tree. 
Because we have no maximum degree condition we will need a class of graphs with lower maximum degree than in Theorem~\ref{main1}. This new class introduces a conflict with trees which are stars. This conflict forces us to pack into $K_{n+1}$ instead of $K_n$ if there are stars present in the sequence of trees. However, the extra vertex does allow several steps in the proof to be less delicate and thus simpler than their counterpart in the proof of Theorem~\ref{main1}.

\begin{proof}%[proof of Theorem~\ref{main2}]
Throughout the proof we will assume that $n$ is sufficiently large for
the appropriate inequalities to hold. Let $t=\frac{1}{10}n^{1/4}$.
We begin by first partitioning the set of trees into three classes.
\begin{enumerate}
\item[(1)] We call $T_i$ \emph{type~I} if there is a set of at most $n^{1/4}$ vertices such that the union of neighborhoods of these vertices contains at least $n^{1/2}$ leaves.
\item[(2)] We call $T_i$ \emph{type~II} if there is a set of at least $n^{1/2}$ independent leaves (and $T_i$ is not type~I).
\item[(3)] We call $T_i$ \emph{path-like} otherwise.
\end{enumerate}

Note that if a tree $T_i$ has a vertex of degree $2n^{3/4}$ then it must be of type I or type II.

\begin{claim}\label{v_and_path}
If $T_i$ is path-like, then $T_i$ contains a path of length $3t+2 = \frac{3}{10}n^{1/4}+2$
such that the internal vertices of the path have degree $2$ in $T_i$.
Furthermore, $T_i$ contains a set of $n^{1/2}$ vertices of degree $2$ such that any two vertices have distance at least $2$ from each other and from the endpoints of the path.
\end{claim}

\begin{proof}\renewcommand{\qedsymbol}{$\blacksquare$}
If $T_i$ is path-like, then $T_i$ has less than $n^{1/2}$ independent
leaves and no set of $n^{1/4}$ vertices whose neighborhood contains $n^{1/2}$ vertices of degree $1$.
Thus if we partition the set of neighbors of the independent leaves into sets of size 
$n^{1/4}$ we see that
the total number of leaves is less than
$n^{1/4}n^{1/2} = n^{3/4}$. 
Furthermore as each vertex $x$ of degree $d(x)>2$ contributes $d(x)-2$ leaves to
$T_i$, we have
$\sum_{\{x : d(x) > 2\}} d(x) -2 < n^{3/4}$. This sum
is at least number of vertices of degree greater than $2$. Thus if we
remove all vertices of degree greater than $2$ from $T_i$ we are left with a
forest of paths with at least $n-n^{3/4}$ vertices and less than
$n^{3/4}$ components. Thus there is a component of size at least
$\frac{n-n^{3/4}}{n^{3/4}} = n^{1/4}-1 > 3t$.

The tree $T_i$ has at most $n^{3/4}$ vertices of degree $1$ and at most
$n^{3/4}$ vertices of degree greater than $2$, thus $T_i$ has at least
$n-2n^{3/4}$ vertices of degree $2$. Thus for $n$ large enough we can
find a path and vertices of degree $2$ as required by the claim.
\end{proof}

For ease of notation put $h = \frac{1}{2}n^{1/2}$ and recall that $t = \frac{1}{10}n^{1/4}$, thus $25t^2 = \frac{1}{4}n^{1/2} = \frac{1}{2}h$. 
For each $i$ define $t'_i$, $t''_i$, and $p_i$ to be the number of type I, type II, and path-like trees, respectively, among $T_1,T_2,\dots, T_{i-1}$.

We partition the vertex set of $K_{n+1}$ into four parts of order $n-h-25t$, $h$, $25t$, and $1$. 
We will refer to these three parts as $K_{n-h-25t}$, $K_h$, $K_{25t}$, and $K_1$.
Now we partition the trees into parts corresponding roughly to the partition of $K_{n+1}$.

\medskip

{\bf Partition of type~I trees:}
For each tree $T_i$ of type~I, there is a set of at most $n^{1/4}$ vertices such that the union of their neighborhoods contains at least $n^{1/2}$ leaves. 
Thus there is a vertex, $x_i$, with $n^{1/4}$ leaves in its neighborhood. 
Let $S_i$ and $Y_i$ be disjoint sets of leaf neighbors of $x_i$ of sizes $3t-(t'_i+p_i)-1$ and $2t$, respectively.
If $T_i$ is not a star, then there is a vertex, $u_i$, different from $x_i$ that has a leaf neighbor, denoted by $v_i$.
If $T_i$ is a star, then let $v_i$ be a leaf neighbor of $x_i$ disjoint from  $S_i$ and $Y_i$.

Define $H_i$ as the union of $x_i$, $u_i$ (if it exists), the set of vertices of
degree greater than $n^{3/4}$, a set of at most $n^{1/4}$ vertices in $T_i$ such that the union of their neighborhoods contains at least
$n^{1/2}$ leaves, and an arbitrary set of vertices in $T_i - S_i- Y_i - v_i$ such that $|H_i| = 25t$.

Partition $T_i$ into six parts: $H_i$, $S_i$, $Y_i$, $v_i$, a set of $h - |S_i|-|Y_i|-1-(i-1)$ neighbors of $H_i$ of degree $1$, denoted by $L_i$, and the remaining 
$n-(i-1)-|H_i|-|S_i|-|Y_i|-1-(h -|S_i|-|Y_i|-1 - (i-1)) =n - h - 25t$ vertices, denoted by $F_i$.

\medskip

{\bf Partition of type~II trees:}
For $T_i$ type~II, let $Y_i$ be a set of $2t$ independent leaves and let $X_i$ be the set of neighbors of $Y_i$ (note that $|X_i|=2t$).
Define $H_i$ as the union of $X_i$, the set of vertices of degree greater than $n^{3/4}$, and an arbitrary set of vertices in $T_i-Y_i$ such that $|H_i| = 25t$.

Partition $T_i$ into four parts: $H_i$, $Y_i$, a set of $h-(i-1) - |Y_i|$ independent leaves that are not adjacent to $H_i$, denoted by $L_i$,
and the remaining $n-(i-1)-|H_i|-|Y_i|-(h - |Y_i| - (i-1)) = n-h-25t$ vertices, denoted by $F_i$.

\medskip

{\bf Partition of path-like trees:}
If $T_i$ is path-like, then let $P_{i}$ be a path on $3t-(t'_i+p_i)$ vertices that is contained in a path on $3t-(t'_i+p_i)+2$ vertices such that all vertices in $P_i$ are degree $2$ in $T_i$.
 Such a path exists by Claim~\ref{v_and_path}.
Let $Y_i$ be a set of $8t$ vertices of degree $2$ that are pairwise of distance at least $2$ from each other and from the endpoints of $P_i$
 and let $X_i$ be the set of neighbors of $Y_i$ (note that $|X_i|=16t$).
 Define $H_i$ to be the union of the set of the two neighbors of the endpoints of $P_i$, the set $X_i$, and an arbitrary set of vertices in $T_i-P_i-Y_i$ such that $|H_i| = 25t$ (note that a path-like tree has no vertex of degree greater than $n^{3/4}$).

Partition $T_i$ into five parts: $H_i$, $P_i$, $Y_{i}$, a set of $h-|P_i|-|Y_i|-(i-1)$ vertices of degree $2$ that are pairwise of distance $2$ from each other and not adjacent to $H_i$ or $P_i$, denoted by $L_i$, and the remaining $n-(i-1)-|H_i|-|P_i|-|Y_i|-(h -|P_i|-|Y_i| - (i-1)) = n-h-25t$ vertices, denoted by $F_i$.

\medskip

{\bf Packing into $K_{n+1}$:}
Observe that for each $i$, $F_i$ is a forest on $n-h-25t$ vertices 
and $\Delta(F_i) \leq n^{3/4} < \frac{10}{3}(n^{3/4}-\frac{1}{2}n^{1/4} - \frac{25}{10}) = \frac{1}{3t}(n - \frac{1}{2}n^{1/2} - \frac{25}{10}n^{1/4}) = \frac{1}{3t}(n-h-25t)$.
Thus by Corollary~\ref{kssz-cor} we can pack the forests into $K_{n-h-25t}$.
In other words there is an edge-coloring of $K_{n-h-25t}$ such that the
edges of color $i$ contain $F_i$.

We can pack each $H_i$ vertex-disjointly into $K_h$ as $\sum_{i=1}^t |H_i| = 25t^2 = \frac{25}{100}n^{1/2} < \frac{1}{2}n^{1/2} = h$.
Because the sets $H_i$ are disjoint in $K_{h}$, for each $i$ the edges $T_i[H_i,F_i]$ in $K_{n+1}$ can be colored with $i$.

For each type~I tree $T_i$, the set $S_i \cup \{x_i\}$ is a star on $3t-(t'_i+p_i)$ vertices and for each path-like tree $T_i$, the set $P_i$ is a path on $3t-(t'_i+p_i)$ vertices.
These stars and paths form a set of trees on $3t,3t-1,\dots, 3t-(t'+p)+1$ vertices (where $t'+p$ is the total number of type~I plus path-like trees). 
By Claim~\ref{stars-paths} these stars and paths pack into a $K_{3t}$ such that the endpoints of the paths are pairwise disjoint in this $K_{3t}$.
Now we embed $K_{3t}$ into $K_h$ in such a way that for each type I tree $T_i$ the center of $S_i \cup \{x_i\}$ in this $K_{3t}$ corresponds to the vertex $x_i$ packed into $H_i$ and the other vertices of the $K_{3t}$ are disjoint from $\cup_{j=1}^t H_j$ in $K_h$. 
For each type~i tree $T_i$ color with $i$ the edges of $S_i\cup \{x_i\}$ in the $K_{3t}$.
Now for each path-like tree $T_i$ color with $i$ the edges of $P_i$ in the $K_{3t}$ and the edge between each endpoint of $P_i$ in the $K_{3t}$ and its neighbor that is in $H_i$. 
Observe that for each type I tree $T_i$, the vertex $x_i$ is incident to at most $i$ edges of color other than $i$. 
Indeed each vertex $x_i$ is incident to at most $t'_i+p_i \leq i-1$ edges of color other than $i$ in the $K_{3t}$ and possibly one more edge if $x_i$ corresponds to the end of a path packed into the $K_{3t}$. Thus for each type I tree $T_i$ let us distinguish two simple cases:

{\bf Case~A:} 
If $T_i$ is not a star, then consider the vertex $x_i$ in $K_{3t}$. By the above step there may be an edge of color $j \neq i$ between $x_i$ and a vertex, denoted by $z$, in $H_j$ (i.e.\ if $x_i$ is the endpoint of a path $P_j$ in the $K_{3t}$). In this case identify the vertex $z$ with $u_i$ and color with $i$ the edge between $u_i$ in $H_i$ and $z$ in $H_j$. If the vertex $x_i$ does not correspond to the end of a path $P_j$, then identify an arbitrary vertex in $K_h - \cup_{i=1}^t H_i - K_{3t}$ with $u_i$ and color with $i$ the edge between $u_i$ and $v_i$.

{\bf Case~B:} 
If $T_i$ is a star, then there is no vertex $u_i$, thus we will identify the vertex $K_1$ with $v_i$ and color the edge between $x_i$ and $v_i$ with color $i$. 

We remark that {\bf Case~B} is the only time when $K_1$ is needed to complete the packing of the trees. Thus if there is no tree which is a star, then we are able to pack into $K_n$.

At this point for each tree $T_i$ the edges of color $i$ in $K_{n+1}$ induce a subgraph of $T_i$ formed from $T_i$ minus some of the vertices of degree $1$ or $2$ (that are pairwise non-adjacent in $T_i$).

We will complete the packing of the trees in three rounds by type in the following order: type~II, path-like, type I. We will pack the trees from largest size to smallest i.e.\ when packing $T_i$ we may assume that all trees of the same type among $T_1,T_2,\dots, T_{i-1}$ have already been packed into $K_{n+1}$. 
A set of vertices in $K_{n+1}$ are called \emph{finished (in color $i$)} if each vertex has as many incident edges of color $i$ as its degree in $T_i$. Otherwise it is \emph{unfinished (in color $i$)}. 

\medskip

{\bf Packing of type II trees:}
Recall that $L_i$ is a partition class of $T_i$ that contains $h - |Y_i| -(i-1)$ independent leaves.
Let $N_i$ be the set of neighbors of $L_i$ in $T_i$, so $|N_i| = h - |Y_i| -(i-1)$. Observe that $N_i \subset F_i$ therefore the vertices $N_i$ are packed into $K_{n-h-25t}$.
To complete the packing of $T_i$ we need to find a matching of uncolored edges between $N_i \cup X_i$ and a set of vertices in $K_{n+1}$ with no incident edge of color $i$. These edges will represent the edges between $N_i \cup X_i$ (which are already packed into $K_{n+1}$) and $L_i \cup Y_i$ (which have not been packed into $K_{n+1}$). Thus coloring the edges of this matching with $i$ will complete the packing of $T_i$.

Consider the bipartite graph between $N_i$ and $K_h - H_i$ in $K_{n+1}$. 
Observe that at this point any color but $i$ have been used on the edges of this bipartite graph and that the edges in a single color class form a forest. 
Thus we have removed at most $t-1$ forests from a complete bipartite graph with class sizes: 
$$|K_h - H_i| = h-25t > 4(t-1)^2$$
and
$$|N_i|  = h - |Y_i| - (i-1) > h -3t > |K_h - H_i| + (t-1).$$
So we may apply Claim~\ref{matching-claim} to the bipartite graph between $N_i$ and $K_h - H_i$ with $t-1$ forest removed.
Therefore there is a matching of uncolored edges between $N_i$ and $K_h - H_i$ that misses only $t-1$ vertices of $K_h - H_i$.
Color the edges of the matching with $i$ such that $2t+(i-1)$ vertices of $K_h - H_i$ are not incident to an edge of color $i$.

Now we consider the set of $2t+(i-1)$ vertices in $K_h - H_i$ not incident to an edge of color $i$.
For each type~I tree $T_j$ that is larger than $T_i$, there
is exactly one vertex $x_j$ in $K_{3t}$.
Thus there is a set of $2t$ vertices in $K_h$ that are not incident to an edge of color $i$ and are disjoint from the vertices $x_j$ for $j < i$.
Identify these $2t$ vertices with $Y_i$ and consider the bipartite graph between $X_i$ and $Y_i$.
Each edge in $X_i \subset H_i$ is incident to at most one edge of each color other than $i$, 
thus the bipartite graph of uncolored edges between $X_i$ and $Y_i$ is the graph obtained by removing at most $t-1$ matchings from a complete bipartite graph. Thus by Claim~\ref{matching2} there is a perfect matching between $X_i$ and $Y_i$. Coloring the edges of this matching with $i$ embeds $Y_i$ into $K_{n+1}$.

Finally, there remains $25t = \frac{25}{10}n^{1/4}$ unfinished vertices in $N_i$. For each color $j \in [i-1]$ there is at most one edge of color $j$ incident to each vertex in $K_{25t}$.
So we have removed at most $i-1$ matchings from a complete bipartite graph with class sizes $25t$.
Thus by Claim~\ref{matching2} there is a perfect matching of uncolored edges between the unfinished vertices of $N_i$ and $K_{25t}$. Coloring this perfect matching with $i$ completes the packing of $T_i$.

\medskip

{\bf Packing of path-like trees:}
Recall that $L_i$ is a partition class of $T_i$ that contains $h-|P_i|-|Y_i|-(i-1)$ vertices of degree $2$ that are pairwise of distance $2$ from each other.
Let $N_i$ be the set of neighbors of $L_i$ in $T_i$. Each vertex in $L_i$ has exactly two neighbors in $N_i$ and no two vertices in $L_i$ share a neighbor in $N_i$, so $|N_i| = 2|L_i|=2(h-|P_i|-|Y_i|-(i-1))$.
Furthermore, each vertex in $L_i$ can be associated with two unique vertices in $N_i$. We call two such vertices in $N_i$ a \emph{pair}.

Recall that $K_{3t}$ only intersects $H_i$ if $T_i$ is type~I, so for a path-like tree $T_i$ we have that $H_i$ and $K_{3t}$ are disjoint.
Now we consider the bipartite graph between $N_i$ and $K_h-H_i-K_{3t}$. At this point any color except $i$ may have been used on the edges between $N_i$ and $K_h-H_i-K_{3t}$
and that the edges in a single color class form a forest.
First let us contract each pair in $N_i$ such that if either of the two edges identified together are already used in the embedding of a tree, then
 the resulting edge is not in the contraction. Denote the contraction of $N_i$ by $N_i'$.
If we contract each pair in $N_i$, then each forest between $N_i'$ and $K_h-H_i-K_{3t}$ becomes the union of two forests.
The complete bipartite graph between $N_i'$ and $K_h-H_i-K_{3t}$ has class sizes:

$$ |K_h-H_i-K_{3t}| = h-25t -3t \geq 4(2(t-1))^2 $$
and
$$
|N_i'| = \frac{1}{2}|N_i| = h -|P_i|- |Y_i| - (i-1) > h -12t >  |K_h-H_i-K_{3t}| + 2(t-1).
$$

The colored edges of this complete bipartite graph form $2(t-1)$ forests, 
so we may apply Claim~\ref{matching-claim} to the bipartite graph of uncolored edges between $N_i'$ and $K_h-H_i-K_{3t}$.
Therefore there is a matching of uncolored edges between $N_i'$ and $K_h-H_i$
that misses only $2(t-1)$ vertices of $K_h-H_i$. 
Now we consider a subgraph of this matching such that exactly $8t+t''_i$ vertices of $K_h - H_i-K_{3t}$ not incident to an edge of the subgraph.
If we return to the uncontracted set $N_i$, then for each edge of the subgraph of the matching we have two edges between a pair in $N_i$ and a vertex
in $K_h - H_i-K_{3t}$. We color these edges with $i$.

Now there are $8t+t''_i$ vertices in $K_h - H_i-K_{3t}$ are not incident to an edge of color $i$. Identify $8t$ of these vertices with $Y_i$ and consider the bipartite graph between $Y_i$ and $X_i \subset H_i$. As before the vertices in $X_i$ can be arranged as pairs of neighbors of degree $2$ vertices of $T_i$. Each edge in $X_i$ is incident to at most $2(t-1)$ edges with colors other than $i$. If we contract the pairs in $X_i$ as before to get $X'_i$, then each vertex in $X'_i$ is incident to at most $4(t-1)$ edges with colors other than $i$.
Thus by Claim~\ref{matching2} there is a perfect matching of uncolored edges between $X'_i$ and $Y_i$. Returning to
the uncontracted set $X_i$, for each edge of the perfect matching we have two edges between a pair in $X_i$ and a vertex in $Y_i$. Color these edges with $i$. Observe that there are $t''_i$ vertices in $K_h- K_{3t}$ not incident to an edge of color $i$ and $t'_i+p_i$ vertices in $K_{3t}$ not incident to an edge of color $i$.
Thus there are exactly $t''_i+t'_i+p_i=i-1$ vertices in $K_h$ not incident to an edge of color $i$.

Finally, there remains $25t = \frac{25}{10}n^{1/4}$ unfinished pairs in $N_i$. For each color $j \in [i-1]$ there are at most two edges of color $j$ incident to each vertex in $K_{25t}$.
If we contract the unfinished pairs of $N_i$ as before, we are left with a complete bipartite graph with $4(t-1)$ matchings removed. 
Thus by Claim~\ref{matching2} there is a perfect matching of uncolored edges between the contracted unfinished pairs of $N_i$ and $K_{25t}$. 
Uncontracting these pairs and coloring the edges corresponding to the perfect matching with color $i$ completes the packing of $T_i$.

\medskip

{\bf Packing of type I trees:}
To complete the packing of $T_i$ we need to color edges incident to $H_i$ that correspond to the $h-(i-1)$ leaf edges removed from $T_i$.
Recall that these leaf edges form a star forest with each center vertex in $H_i$. Observe that each vertex in $H_i$ is incident to at
most $2$ edges of each color other than $i$.
Therefore each vertex in $H_i-x_i$ is incident to at least $|(K_h - H_i) \cup K_{25t}| - 2(t-1) = h-25t+25t-2(t-1) = h-2(t-1)$ uncolored edges with an endpoint in $(K_h - H_i) \cup K_{25t}$.
So we may apply Claim~\ref{stars-claim} with $k=2(t-1)$ to the bipartite graph of uncolored edges between $H_i-x_i$ and $(K_h - H_i) \cup K_{25t}$
to find the appropriate star forest removed from $T_i$. Coloring this star forest with $i$ finishes each vertex in $H_i-x_i$.
Now there are $2(t-1)+t''_i$ vertices in $K_h-K_{3t}$ that are not incident to an edge of color $i$. We can identify $2(t-1)$ of these vertices with $Y_i$
and color the edges between $x_i$ and $Y_i$ to complete the packing of $T_i$.
\end{proof}

\begin{proof}[Proof of Proposition~\ref{main2-cor}]
We repeat the proof of Theorem~\ref{main2} and observe that if there is no star in the set of trees, then {\bf Case B} above never occurs. This is the only situation when the vertex in $K_1$ is used. Thus we are only packing the trees into $K_n$ which is what is claimed by Proposition~\ref{main2-cor}. 
\end{proof}

\section*{Acknowledgments}
We would like to thank Hong Liu and Bernard Lidick\'y for a careful reading of the manuscript.

\bibliographystyle{acm}
\bibliography{treepackingbib}

\end{document}